\documentclass{amsart}
\usepackage{amsmath,amssymb,amscd,url,color}

\usepackage[T1]{fontenc}
\usepackage{amssymb, amsthm, amsmath}
\usepackage[english]{babel}
\usepackage[latin1]{inputenc}

\usepackage{hyperref}

\newcommand{\m}{\ensuremath}

\newcommand{\fC}{\mathfrak{C}}

\newcommand{\setR}{\mathbb{R}}

\newcommand{\setN}{\mathbb{N}}

\newcommand{\setQ}{\mathbb{Q}}

\newcommand{\lam}{\m{\lambda}}
\newcommand{\om}{\m{\omega}}

\newcommand{\al}{\m{\alpha}}

\newcommand{\ph}{\m{\varphi}}
\newcommand{\ps}{\m{\psi}}

\newcommand{\then}{\m{\Rightarrow}}

\newcommand{\set}[1]{\ensuremath{\{#1\}}}
\newcommand{\seq}[1]{\ensuremath{\langle#1\rangle}}
\newcommand{\lseq}[3]{\ensuremath{\langle #1_{#2} : #2 < #3\rangle}}
\newcommand{\inseq}[3]{\ensuremath{\langle #1_{#2} : #2 \in #3\rangle}}

\newcommand{\card}[1]{\ensuremath{|#1|}}

\DeclareMathOperator{\tp}{tp}

\DeclareMathOperator{\lev}{lev} \DeclareMathOperator{\rkdp}{rk-dp}
\DeclareMathOperator{\alt}{alt}

\DeclareMathOperator{\tS}{S}

\newtheorem{theorem}{Theorem}[section]

\newtheorem{claim}[theorem]{Claim}

\newtheorem{corollary}[theorem]{Corollary}

\newtheorem{fact}[theorem]{Fact}
\newtheorem{lemma}[theorem]{Lemma}

\newtheorem{proposition}[theorem]{Proposition}

\newtheorem{observation}[theorem]{Observation}

\theoremstyle{definition}
\newtheorem{definition}[theorem]{Definition}
\newtheorem{example}[theorem]{Example}
\newtheorem{remark}[theorem]{Remark}
\newtheorem{question}[theorem]{Question}

\theoremstyle{remark}
\newtheorem{rmk}[theorem]{Remark}
\newtheorem{clm}{Claim}[theorem]

\newcommand{\0}{\emptyset}

\renewcommand{\phi}{\varphi}

\long\def\symbolfootnote[#1]#2{\begingroup%
\def\thefootnote{\fnsymbol{footnote}}\footnote[#1]{#2}\endgroup}

\makeatletter

\def\Ind#1#2{#1\setbox0=\hbox{$#1x$}\kern\wd0\hbox to 0pt{\hss$#1\mid$\hss}
\lower.9\ht0\hbox to 0pt{\hss$#1\smile$\hss}\kern\wd0}
\def\ind{\mathop{\mathpalette\Ind{}}}
\def\Notind#1#2{#1\setbox0=\hbox{$#1x$}\kern\wd0\hbox to 0pt{\mathchardef
\nn=12854\hss$#1\nn$\kern1.4\wd0\hss}\hbox to
0pt{\hss$#1\mid$\hss}\lower.9\ht0 \hbox to
0pt{\hss$#1\smile$\hss}\kern\wd0}
\def\nind{\mathop{\mathpalette\Notind{}}}

\def\thind{\mathop{\mathpalette\Ind{}}^{\text{\th}} }
\def\nthind{\mathop{\mathpalette\Notind{}}^{\text{\th}} }
\def\uth{{\rm{U}^{\text{\th}}}}
\def\ur{{\rm{U}}}

\def\po{\le_p}

\begin{document}

\newcommand{\Av}{\operatorname{Av}}

\newcommand{\cL}{\mathcal L}

\title{Additivity of the dp-rank}
\author{Itay Kaplan}

\address{Itay Kaplan \\
Mathematischen Instituts\\
Universit\"{a}t M\"{u}nster\\
Einsteinstra{\ss}e 62\\
48149 M\"{u}nster\\
Germany}

\author{Alf Onshuus}
\address{ Alf Onshuus \\
Departamento de Matem\'aticas \\
Universidad de los Andes \\
Cra 1 No. 18A-10, Edificio H \\
Bogot\'a, 111711 \\
Colombia}

\urladdr{ http://matematicas.uniandes.edu.co/aonshuus}

\author{Alexander Usvyatsov}
\address{ Alexander Usvyatsov\\ Universidade de Lisboa \\
  Centro de Matem\'{a}tica e Aplica\c{c}\~{o}es Fundamentais\\
  Av. Prof. Gama Pinto,2\\
  1649-003 Lisboa \\
  Portugal}

\urladdr{http://ptmat.fc.ul.pt/\textasciitilde alexus}

\thanks{The third author was supported by FCT grant SFRH / BPD / 34893 /
2007 and FCT research project PTDC/MAT/101740/2008.}

\thanks{We thank the organizers of the meeting ``Model theory: Around
Valued Fields and Dependent Theories'' at Oberwolfach in January
2010. Special thanks to Sergei Starchenko for valuable discussions
during and after the meeting. We would also like to thank John
Goodrick, Alfred Dolich, Pierre Simon and the anonymous referee
for some very helpful remarks and comments on earlier drafts of
this paper.}



\def\Ind#1#2{#1\setbox0=\hbox{$#1x$}\kern\wd0\hbox to 0pt{\hss$#1\mid$\hss}
\lower.9\ht0\hbox to 0pt{\hss$#1\smile$\hss}\kern\wd0}
\def\ind{\mathop{\mathpalette\Ind{}}}
\def\Notind#1#2{#1\setbox0=\hbox{$#1x$}\kern\wd0\hbox to 0pt{\mathchardef
\nn=12854\hss$#1\nn$\kern1.4\wd0\hss}\hbox to
0pt{\hss$#1\mid$\hss}\lower.9\ht0 \hbox to
0pt{\hss$#1\smile$\hss}\kern\wd0}
\def\nind{\mathop{\mathpalette\Notind{}}}

\def\thind{\mathop{\mathpalette\Ind{}}^{\text{\th}} }
\def\nthind{\mathop{\mathpalette\Notind{}}^{\text{\th}} }
\def\ur{\text{U}}
\def\uth{\text{U}^{\text{\th}} }
\def\po{\leq_p}
\def\Cb{\text{Cb}}
\def\fC{\mathcal{C}}

\begin{abstract}
The main result of this article is sub-additivity of the dp-rank.
We also show that the study of theories of finite
dp-rank can not be reduced to the study of its dp-minimal types, and
discuss the possible relations between dp-rank and VC-density.
\end{abstract}

\maketitle

\section{introduction}\label{1}

This paper grew out of discussions that the authors had during a
meeting in Oberwolfach in January 2010, following a talk of
Deirdre Haskell, and conversations with Sergei Starchenko, on their
recent joint work with Aschenbrenner, Dolich and Macpherson
\cite{ADHMS}. Haskell's talk made it apparent to us that the
notion of VC-density (Vapnik-Chervonenkis density), investigated
in \cite{ADHMS}, is closely related to ``dependence rank''
(dp-rank) introduced by the third author in
\cite{Us}. Discussions with Starchenko helped us realize that
certain questions, such as additivity, which were (and still are,
to our knowledge) open for VC-density, may be approached more
easily in the context of dp-rank. This paper is the first step in
the program of investigating basic properties of dp-rank and its
connections with VC-density.

Whereas dp-rank is a
relatively new notion, VC-density and related concepts have been
studied for quite some time in the frameworks of machine learning,
computational geometry, and other branches of theoretical computer science.
Recent developments
point to a connection between
VC-density and dp-rank, strengthening the bridge between model theory and
these subjects. We believe that
investigating properties
of dp-rank is important for discovering the nature of this connection.
Furthermore, once this relation is better understood,
theorems
about dp-rank are likely to prove useful in the study
of finite and infinite combinatorics related to VC-classes.

\bigskip

Dp-rank was originally defined in \cite{Us} as an attempt to
capture how far a certain type (or a theory) is from having the
independence property. It also helped us to isolate a
\emph{minimality notion} of dependence for types and theories
(that is, having rank 1). We called this notion dp-minimality and investigated it in
\cite{OnUs}. Both dp-rank and dp-minimality were simplifications
of Shelah's various ranks from \cite{Sh863}, and appropriate
minimality notions. Our simplified notion of rank turned out to be
very close to Shelah's $\kappa_{\text{ict}}$, but localized to a
type.

A few characterizations of dp-minimality have since been given in
the literature, and similar equivalences can also be applied to
higher ranks. It seems clear at this point that part of the
strength of this concept is the interaction between those
equivalent definitions. We are mainly referring to the
``standard'' (syntactic) independent array definition as given in
\cite{Us,OnUs2} (see Definition \ref{dfn:strongdep}), which is very
useful when one wants to deal with formulas, and a very simple
``semantic'' variant (see Definition \ref{dp-rank}), which can be
found, for the dp-minimal case, in Simon's paper (\cite{Si}), but
which (as far as we know) has never been stated for a type with
rank greater than 1. In this paper, we will mainly work with the
semantic definition, which proves useful and convenient for our
purposes. However, throughout the paper we prove quite a few
different characterizations of dp-rank, especially in the finite
rank case; we summarize them in Theorem \ref{summary}.

\medskip

As the word ``rank'' indicates, dp-rank is a certain measure of
the ``size'' of a type. It would probably be more accurate to say
that dp-rank measures ``diversity'' of realizations of the type --
how much the realizations differ from each other, as can be seen
by external parameters. We will elaborate on this below. Just as
with most rank notions, one wonders whether it has basic
properties such as sub-additivity: the rank of a tuple should be
bounded by the sum of the ranks of its elements. This is the main
result of Section \ref{2}. As a corollary, we obtain global bounds
for alternation ranks of formulas in a theory of finite rank (so
in particular in a dp-minimal theory).

\medskip


Dp-rank is a notion which is related to weight in stable theories
(and motivated by it), and to certain more recent notions of
weight in dependent theories, e.g. \cite{OnUs}, \cite{OnUs2}, and
\cite{KaUs}. It is therefore natural to wonder whether dp-rank can
play a role similar to weight in dependent theories. Although this
is still a line of work which can be pursued (and probably
provides natural lines of research), in this paper we give
limitations to the analogy, showing that dp-rank fails to have
certain properties that one would hope for in a notion of weight,
in at least two significant ways.

One of our original intentions (influenced greatly by the weight
analogy) was to prove results for dp-minimal types, and then try
to extend them to higher ranks by induction, at least for the
finite rank case. 
The analogy with weight in stable (and even simple) theories led
us to ask whether every type of finite dp-rank can be
``analyzed'' to some extent by types of rank 1. This turns out to
be not the case, and provides the first limitation to the weight
comparison. Example \ref{example} presents a theory with types of
finite rank but \emph{no dp-minimal types}. 
This is a very
``tame'' strongly dependent theory (hence of ``finite weight'' -- in fact, of dp-rank 2),
but there are no types of ``weight'' 1 (whatever definition of
weight one chooses to use; see e.g. \cite{OnUs}).

Another direction where the analogy fails is the finite/infinite
correspondence. Hyttinen's results in \cite{Hy}, that yield a
decomposition a type of finite weight in a stable
theory into types of weight 1, also imply that no type has
infinite but ``rudimentarily finite'' weight. In other words, if a
type has weight $\aleph_0$, then there exists an infinite set
witnessing this. In Subsection \ref{sub:example} (particularly,
Theorem \ref{thm:example}) we observe that this fails to be the
case for dp-rank, even in stable theories. The example presented
there also answers negatively an analogous question concerning
Adler's notion of burden in \cite{Ad2} (since burden is
essentially the same as dp-rank in dependent theories\footnote{If
either is finite then the two coincide; for infinite cardinals the
dp-rank of a type may be the successor of the burden of the
type.}). Specifically, this is an example of an $\aleph_0$-stable
theory (in particular, superstable) such that every type over
$\emptyset$ has infinite dp-rank. Since in a superstable theory
weight of any given complete type is finite, Theorem
\ref{thm:example} exemplifies a very important difference between
weight and dp-rank: weight only takes into account nonforking
extensions of the type; so when passing to a forking extension,
weight may grow. Dp-rank does not distinguish between different
kinds of extensions, so it can only go down when the set of
parameters is increased (as is the case with most rank/dimension
notions).

\medskip

In a sense, with the notion of dp-rank, we capture a very
particular aspect of the size of the type, and our thesis is that
it might be ``better'' to think of dp-rank in relation to finite
combinatorial invariants of Vapnik-Chervonenkis (VC) classes,
particularly, VC-density, instead of in terms of weight.

%

VC-density is one of the possible measures of growth rate of types
over finite sets. Specifically, assuming that $\Delta(x,y)$ is a
finite set of dependent formulas (which is equivalent to the
growth rate of the number of finite $\Delta$-types being
polynomial rather than exponential), VC-density bounds the degree
of the corresponding polynomial. So VC-density 1 corresponds to
linear growth of the number of $\Delta$-types over finite sets. We
prove in Section \ref{3} that one way of looking at dp-rank is the
following: dp-rank of a type $p$ corresponds to the degree of a
polynomial that measures how many types can a realization
$c\models p$ realize over finite \emph{indiscernible sequences}.
In other words, dp-rank of $p$ is a measure of the number of
realizations of $p$ that have different types over finite
indiscernible sequences.

%
%
%

So in particular dp-minimality means that realizations of $p$
realize only order of $n$ types over any given indiscernible
sequence of length $n$. This is equivalent, as is shown in Section
\ref{sub:dp_rank_and_counting_finite_types}, to the following:
subsets of indiscernible sequences that a realization of $p$ can
definably pick are simply intervals. These intervals can be, of
course, of length 1 (that is, singletons). The important
observation here is that a dp-minimal element can not pick
(definably) a finite tuple of size bigger than 1 from an arbitrary
indiscernible sequence. Or, more precisely, its alternation rank
can not be bigger than 2. In this respect the behavior of a
dp-minimal element resembles an element in either an o-minimal
theory, or the theory of equality.

A number of distinctions between the notions of VC-density and
dp-rank should be pointed out. The obvious ones are that we only
count number of types over indiscernible sequences (as opposed to
arbitrary finite sets), and restrict ourselves to realizations of
a given type (the second one is not important -- one can look at
dp-rank of a partial type as well). Another difference is that,
when calculating dp-rank, we do not restrict the set of formulas
that one is allowed to use to a finite set. This is why, even in a
strongly dependent theory, one can end up with types of dp-rank
$\omega$ (one obtains more and more types over indiscernible
sequences by changing the formulas), one more thing that Theorem
\ref{thm:example} exemplifies.

However, dp-rank is a very natural model theoretic analogue of VC-density.
Part of the strength of model theoretic techniques is the ability to approximate complex
phenomena in better behaved structures. Indiscernible sequences have
already proved very helpful for such approximations in various contexts.
Seeking connections between VC-density and dp-rank is another implementation of this idea.

\bigskip

Since we work with the semantic definition of dp-rank, we need
some technical results on indiscernible sequences which are quite
interesting on their own. In Section
\ref{sec:extending_indiscernible_sequences} we prove a proposition
(Proposition \ref{unique complete 2}) which provides a
``consistent'' way to extend indiscernible sequences in an
arbitrary theory. Specifically, in the proofs in Section \ref{2}
we are sometimes faced with the following situation: a sequence
$I$ is indiscernible over various subsets of a set $B$, but not
(necessarily) over $B$, and we would like a uniform way of
extending $I$ to a longer sequence with the same properties. If
the theory is assumed to be dependent (and $I$ is unbounded) one
can just take the average type of $I$ over $B$. However, we are
not assuming dependence, and it seems of independent interest to
find a general technique allowing this (and more) in an arbitrary
theory. Here Shelah's general notion of
average type with respect to an ultrafilter comes in handy. 

%



\subsection{Structure of the paper}

We begin Section \ref{1.5} with definitions, characterizations, and
basic properties. Then we proceed to a few examples, which point
out what one can and can not expect from dp-rank.

In Section \ref{sec:extending_indiscernible_sequences} we develop
a consistent way of extending mutually indiscernible sequences,
which serves us in Section \ref{2}.


In Section \ref{2} we will prove the main result of the paper, the
sub-additivity of the dp-rank. We start with the dp-minimal case,
then proceed to types of finite rank, and finally the infinite
case. There was no need, in fact, to give a separate proof for
dp-minimal types: the general finite case can be easily modified
to include rank 1 as well. However, the proof for dp-minimal types
is much less involved, so we include it in order to exemplify the
general principles that are used, and then proceed to the finite
case by induction. The proof for infinite ranks is different, but
easier -- infinite combinatorics is more flexible, and calls for
much less precise computations. We conclude the section with some corollaries,
such as global bounds
on alternation ranks of formulas.

In Section \ref{3} we discuss what is known about the relation
between VC-density and dp-rank. Theorem \ref{summary} summarizes all
the main equivalent ways of looking at dp-rank. We also pose a few questions and
set up a framework for future work.

\bigskip

Throughout the paper, we will work in a monster model of an
arbitrary first order theory $T$. In particular, although the main
object of study is a ``dependence rank'', at no point do we assume
that $T$ is dependent.

We will not distinguish between singletons and finite tuples in
the notation. For example, when writing $\ph(x)$, or ``$a$ is in
the sort of $x$'', if not specified otherwise, $x$ and $a$ could
be tuples.

\section{Dp-rank: definitions and basic properties}\label{1.5}\label{sub:dp_rank_and_counting_finite_types}

We begin with the definition of the main notion investigated in this paper.


\begin{definition}\label{dp-rank}
Let $p(x)$ be a partial (consistent) type over a set $A$. We
define the dp-rank of $p(x)$ as follows.

\begin{itemize}

    \item The dp-rank of $p(x)$ is
    always greater or equal than 0. Let $\mu$ be a cardinal. We will say that $p(x)$ has dp-rank $\le \mu$
    (which we write $\rkdp(p)\le\mu$) if given any realization $a$ of $p$ and any $1+\mu$
    mutually $A$-indiscernible sequences, at least one of them is
    indiscernible over $Aa$.


\item We say that $p$ has dp-rank $\mu$ (or
$\rkdp(p)=\mu$) if it has dp-rank $\le \mu$, but it is not
the case that it has dp-rank $\le \lam$ for any
$\lam<\mu$.


\item We call $p$ \emph{dp-minimal}  if it has dp-rank 1. 

\item We call $p$ \emph{dependent} if dp-rank of $p$ 
exists, that is, if it is an ordinal. In this case we write
$\rkdp(p)<\infty$. Otherwise we write $\rkdp(p)=\infty$.

\item We call $p$ \emph{strongly dependent} if
$\rkdp(p)\le \om$.

\end{itemize}


\end{definition}

\begin{rmk}
	It may seem that dp-rank of $p$ depends on the set $A$; however, we will see in Remark \ref{rem:doesnotdepend} below that this is not the case. Hence all the notions in Definition \ref{dp-rank} above are well defined. 
\end{rmk}

%

The following is easy and standard.

\begin{remark}
The following hold for any (partial) type $p(x)$ and a set $A$.
\begin{enumerate}
\item
$p$ has dp-rank 0 if and only if it is algebraic.

\item
$p$ is dependent if and only if the $\rkdp(p)\le |T|^+$.
\item
If $p$ is a type over a set $A$, and $p'$ is a type over a set $B\supseteq A$ that extends $p$, then 
$\rkdp(p')\le\rkdp(p)$. 
\end{enumerate}
In particular, if $p$ is dp-minimal, then any
extension of it is either dp-minimal or algebraic.
\end{remark}

Definition \ref{dp-rank} is a nice semantic characterization of
the notion of dp-rank. We find it more convenient for the purposes
of this paper than the syntactic definition in \cite{OnUs2}. It is
also much easier to grasp, in case one is unfamiliar with the
concept. However, when working with formulas, it is useful to have
a more syntactic notion, and we would like to prove that our
``soft'' characterization is equivalent to the original one. In
case the reader is unwilling to deal with technical concepts, it
is possible to skip the following definition and Proposition
\ref{prp:equivalence} in the first reading. These will not be used
almost at all in the main body of the paper (Sections
\ref{sec:extending_indiscernible_sequences} and \ref{2}), but they
are key for understanding important characterizations of dp-rank
in the finite case and the connection to VC-density (Proposition
\ref{thm:rank_counting} and Section \ref{3}), as well as the fact that dp-rank 
of a type does not depend on the set over which it is calculated. 

The following definitions
were motivated by the original definition of strong dependence by
Shelah (see e.g. \cite{Sh863}) and appear in \cite{Us} and
\cite{OnUs}.

\begin{definition}\label{dfn:strongdep}
$\left.\right.$
    A \emph{randomness pattern} of depth $\kappa$ for a (partial) type $p(x)$ over a set $A$ is an
    array $\langle b_i^\alpha \colon i<\omega\rangle_{\alpha<\kappa}$ and formulae $\phi_\alpha(x,y_\alpha)$ for
    $\alpha<\kappa$ such that:
    \begin{enumerate}
    \item
        the sequences $I^\alpha = \langle b^\alpha_i\rangle_{i<\omega}$ are
        mutually indiscernible over $A$; that is, $I^\alpha$ is
        indiscernible over $AI^{\neq \alpha}$,
    \item
        $length(b^\alpha_i) = length(y_\alpha)$,
    \item
        for every $\eta \in {}^\kappa\omega$, the set
        $$ \Gamma_\eta = \{\phi_\alpha(x,b^\alpha_\eta)\}_{\alpha < \kappa} \cup
        \{\neg\phi_\alpha(x,b^\alpha_i)\}_{\alpha<\kappa, i<\omega, i\neq \eta(\alpha)}$$
        is consistent with $p$.
    \end{enumerate}
	We will omit $A$ if it is clear from the context.
\end{definition}


\smallskip
The following is standard. 

\begin{lemma}\label{lem:randomness}
	Let $p$ be a (partial) type over a set $A$. Then there is a randomness pattern of depth $\kappa$ 
	over $A$ if and only if there exists an array $\langle b_i^\alpha \colon i<\omega\rangle_{\alpha<\kappa}$ and formulae $\phi_\alpha(x,y_\alpha)$ for
    $\alpha<\kappa$ that satisfy clauses (ii) and (iii) of Definition \ref{dfn:strongdep}.

\end{lemma}
\begin{proof}
	Without loss of generality $A=\emptyset$ (note that $A$ contains the domain of $p$).
	Given an array satisfying clauses (ii) and (iii) of Definition \ref{dfn:strongdep}, we 
	may assume by compactness that the all the $\kappa$ sequences $\langle b_i^\alpha  \rangle$ are as long as we wish, $\langle b_i^\alpha \colon i<\lambda \rangle$ for $\lambda$ arbitrarily big. Applying Ramsey's Theorem and compactness, one obtains new sequences for which clause (i) of the definition of a randomness pattern holds as well. 
\end{proof}

One consequence of the Lemma is the following: let $p$ be a type, and let $A,B$ be sets, both containing the domain of $p$. Then there is a randomness pattern of depth $\kappa$ for $p$ over $A$ if and only if there exists such a pattern over $B$.

\medskip

In \cite{OnUs2} we defined dp-rank of a type $p(x)$ as the
supremum of all $\kappa$ such that there is a randomness pattern
of depth $\kappa$ for $p(x)$. The following Proposition shows that
Definitions \ref{dp-rank} above are equivalent to the original
ones.

The first appearance of any such equivalence appeared for the
dp-minimal case in Lemma 1.4 of \cite{Si}. We do not know that
anyone has generalized this even for finite dp-ranks (or
randomness patterns of finite depth).




\begin{proposition}\label{prp:equivalence}
The following are equivalent for a complete type $p(x)$ over $A$.
Notice that $\kappa$ below may be a finite cardinal.
\begin{enumerate}

\item There is a randomness pattern of depth $\kappa$ for $p(x)$
over $A$.

\item It is not the case that the dp-rank of $p(x)$ is less than
or equal to $\kappa$.

\item There exists a set $\mathcal I:=\{I^j\}:=\{\langle
a_i^j\rangle_{i\in I} \mid j\in\kappa\}$ of $\kappa$ infinite
mutually indiscernible sequences over $A$, and a realization $c$
of $p(x)$ such that for all $j$ there are $i_1,i_2$ such that
$\tp(a_{i_1}^j/Ac)\neq \tp(a_{i_2}^j/Ac)$.

\end{enumerate}
\end{proposition}

\begin{proof}
The proof of the equivalence of (ii) and (iii) in \cite{Si}
works exactly for the general case. Also, it is clear that (i)
implies (iii). 

\medskip

Assume (iii) and we will prove (i). Let $\mathcal I_j:=\langle
a_i^j\rangle$ and we will assume that all such $I^j$ are indexed
by $\mathbb Z$, and let $\phi^j$ be the formula such that
$\phi^j(c, a_{i_1})$ and $\neg \phi^j(c,a_{i_2})$ holds.

\begin{clm}
If $\{i\mid \phi^j(c, a^j_{i})\}$ is both coinitial and cofinal,
then there is a subsequence $I_0^j$ of $I^j$ such that (after
re-enumerating the elements) $I_0^j:=\langle
a^{0,j}_i\rangle_{i\in \mathbb Z}$ and $\neg \phi^j(c,a^{0,j}_i)$
holds if and only if $i=0$.
\end{clm}

\begin{proof}
The construction of $I_0^j$ is immediate from the definition.
\end{proof}

Notice that if in every sequence $I^j$ we have that either $\{i
\mid \phi^j(c, a^j_{i})\}$ or $\{i\mid \neg \phi^j(c, a^j_{i})\}$
is coinitial and cofinal, then we can replace $I^j$ by the
subsequence $I_0^j$ and (replacing $\phi^j(x,y)$ for $\neg
\phi^j(x,y)$ if necessary) we would have an instance of (i).

Now, if for some $j$ we have that both $\{i \mid \phi^j(c,
a^j_{i})\}$ and $\{i\mid \neg \phi^j(c, a^j_{i})\}$ are not
coinitial and not cofinal, then we have that, for example, $\{i
\mid \phi^j(c, a^j_{i})\}$ is coinitial and $\{i\mid \neg
\phi^j(c, a^j_{i})\}$ is cofinal. In this case, if we define
$\phi^j_0:=\phi^j(x,y_1)\wedge \neg \phi^j(x,y_2)$ and $I^j_0$ as
a sequence $\langle a^j_{2i},a^j_{2i+1}\rangle_{i\in \mathbb Z}$,
we would preserve the mutual indiscernibility, we would have
instances of both $\phi^j_0(c,a^j_{2i},a^j_{2i+1})$ and $\neg
\phi^j_0(c,a^j_{2i},a^j_{2i+1})$, and $\{ i\mid \neg
\phi^j(c,a^j_{2i},a^j_{2i+1})\}$ would be coinitial and cofinal.
Applying the claim to all such sequences we would have an instance
witnessing (i).
\end{proof}

\begin{remark}\label{rem:doesnotdepend}
	It follows from Lemma \ref{lem:randomness} and the equivalence of (i) and (ii) in Proposition \ref{prp:equivalence} that 
	dp-rank of a type \emph{does not depend on the set $A$} over which it is computed. 
	In other words, the notion $\rkdp(p)$ in Definition \ref{dp-rank} is well-defined. 
	%
\end{remark}



\medskip

In the finite rank case, we can prove another characterization of dp-rank, in terms of a natural generalization
of the notion of alternation rank.
Given a formula $\varphi(x,y)$, we define
the $p$-alternation rank of $\varphi(x,y)$ over $A$ as follows:
$\alt^{p(x)}_A(\varphi(x,y))\ge k$ if there exists an
$A$-indiscernible sequence $I$ and $c\models p$ such that the
truth value of $\varphi(c,y)$ has $k$ alternations in $I$. The
$p$-alternation rank of $\varphi(x,y)$ over $A$ is the maximal $k$
(if exists) such that $\alt^{p(x)}_A(\varphi)\ge k$. As usual, if
$p\in\tS(A)$, we may omit $A$.

\begin{proposition}\label{thm:rank_counting}
    The following are equivalent for a partial type $p(x)$ over $A$ and $k<\omega$:
    \begin{enumerate}
        \item $\rkdp(p) \ge k$.
        \item There exists a formula $\ph(x,y)$ and an $A$-indiscernible sequence $I$ in the sort of $y$  such that for every subset $I' \subseteq I$ of size $k$,  there is a $c\models p$ such that $\ph(c,y)\cap I=I'$.
        \item There exists a formula $\varphi(x,y)$ such that $\alt^{p(x)}_A(\ph(x,y))\ge
        2k$.
    \end{enumerate}
\end{proposition}

All the indiscernible sequences in the Proposition are presumed to be
infinite. Notice also that what (ii) is essentially saying is that 
every $k$-tuple in $I$ is $\ph(c,y)$-definable for some $c\models
p$.

\begin{proof}
    (i) $\implies$ (ii).
We will use the syntactic definition here. If $\rkdp(p)\ge k$,
then by Proposition \ref{prp:equivalence} there is a randomness
pattern of depth $k$ for $p(x)$; that is, there are $c\models p$,
formulas $\ph_1(x,y_1),\ldots, \ph_k(x,y_k)$ and $A$-mutually
indiscernible sequences $I_i = \seq{a^i_j \colon j<\om}$ in the
sort of $y_i$ (for $i=1, \ldots, k$), such that for every
$j_1,\ldots,j_k$ there is $c = c_{j_1,\ldots,j_k}\models p$ such
that $\ph_i(c,a^i_j)$ if and only if $j=j_i$.

Let $\ph(x,y_1\ldots y_k)=\bigvee_{i=1}^k\ph_i(x,y_i)$, and let
$I$ be an $A$-indiscernible sequence in the sort of $y_1\ldots
y_k$ defined as follows: $I = \lseq{a}{j}{\om}$, where
$a_j=a^1_j\ldots a^k_j$.

Choose arbitrary $k$ distinct indices $j_1, \ldots, j_k$. It is
easy to see that $c = c_{j_1,\ldots,j_k}$ is a realization of $p$
such that $\ph(c,a_j)$ if and only if $j\in \set{j_1,\ldots,j_k}$,
as required in (ii).

%

    (ii) $\implies$ (iii) is trivial.

(iii) $\implies (i)$. Let $I$ be an $A$-indiscernible sequence and
$c\models p$ such that $\ph(c,y)$ has $\ge 2k$ alternations in
$I$. We may assume that the number of alternation is finite, and
in fact equal to $2k$.
    We prove by induction on $k$ the following claim:

\smallskip
\noindent \emph{Claim.} If $I$ is an $A$-indiscernible sequence of
order type $\setQ$, and $\ph(c,y)$ has $2k$ alternations in $I$,
then there is a randomness pattern $(I^\al,\ph^\al)$ of depth $k$
for $\tp(c/A)$ with $I^\al$ being segments of $I$, and
$\ph^\al=\ph$ or $\ph^\al=\neg\ph$ for all $\al$.

\smallskip
The base case $k=0$ is trivial. So now given $I =
\inseq{a}{q}{\setQ}$ and $c$ as in the claim, since the order type
of $I$ is $\setQ$, there is an infinite initial segment on which
$\ph(c,y)$ is constant; assume for example that $\ph(c,y)$ holds,
and let $r \in \setR$ be the minimal cut such that for any $q> r$,
the truth value of $\ph(c,y)$ has changed signs twice in the
interval $(-\infty, q)$. Let $q> r$ be such that $\ph(c,a_q)$
holds, and there is no sign change between $r$ and $q$. Let $I^0 =
\seq{a_{<r}}^\frown\seq{a_{(r,q)}}$ and $\ph^0 = \ph$. Notice that
by indiscernibility and compactness, given any $a_i\in I^0$ the
type
\[
\neg \phi(x,a_i)\wedge \bigwedge_{j\in I^0\setminus \{a_i\}}
\phi(x,a_j)
\]
is consistent.

By the induction hypothesis, since the order type of
$I_{>q}=\seq{a_{>q}}$ is still $\setQ$, and since $I_{>q}$ has at
least $2k-2$ alternations of $\ph(c,y)$, we can find a randomness
pattern $\{ I^\al \} ,\{\ph^\al\}$ for $c$ over $A$, for
$\al=1,\ldots,k-1$ and $I^\al$ segments of $I_{>q}$. Now clearly
$\{I^\al\}_{\al=0}^{k}$ are mutually $A$-indiscernible (since they
are all segments of the same $A$-indiscernible sequence), and it
is easy to see that $\{I^\al\}_{\al\leq k} ,\{\ph^\al\}_{\al\leq
k}$ is a randomness pattern for $c$ over $A$ of depth $k$.

This finishes the proof of the claim, and the theorem.


\end{proof}


\begin{corollary}\label{cor:dprank_alt}
Assume that in a theory $T$ every type $p(x)$ over $\0$ in the sort of $x$
has dp-rank $\le k$. Then for every formula $\ph(x,y)$, the
alternation rank of $\ph(x,y)$ is bounded by $2k+1$.
\end{corollary}
\begin{proof}
If $\alt(\ph(x,y))\ge 2k+2$, then there is an indiscernible
sequence in the sort of $y$ and some $c$ in the sort of $x$, that
witness this. By (iii) \then (i) in Proposition
\ref{thm:rank_counting}, $\rkdp(\tp(c/\0))\ge k+1$.
\end{proof}

\bigskip

We conclude this section with a a few examples that illustrate
things which can not be expected from dp-rank. We will leave some
of the technical details of the examples to the reader.

\subsection{Theory of dp-rank 2 with no dp-minimal types}


As we mentioned above, one might be drawn to think that all the
study of theories of finite dp-rank (theories where all dp-ranks
are finite) can be reduced to dp-minimal types. This, however is
not the case.

\begin{example}\label{example}
Consider the theory of an infinite set with two dense linear orders
$<_1$ and $<_2$, and take the model completion of it. This is, the
theory of a structure $M$ in $\mathcal L:=\{<_1, <_2\}$ such that
any finite formula consistent with $<_1$ and $<_2$ being dense
linear orders is realized in $M$.

Every one type has dp-rank 2, and there are no dp-rank 1 types:

First of all, the model completion exists and by definition it is
complete and has elimination of quantifiers in the language
$\mathcal L:=\{<_1, <_2\}$, so $\tp(a/A)$ can be understood by
formulas of the form $x<_1 a$,
$x<_2 a$,
and $x=a$ for suitable choices of $a\in A$.

It is not hard to show now that given any set $A$, any
1-variable type $p(x)\in \text{S}(A)$ has dp-rank 2. Given any such set and type,
it is enough to find mutually $A$-indiscernible sequences $\langle a_i
\rangle$ and $\langle b_j \rangle$ such that for every $k,\ell$ we have:
\begin{itemize}
\item
The set $$p(x)\cup \{x>_1a_i\}_{i\le k} \cup \{x<_1a_i\}_{i>k}\cup \{x>_2b_j\}_{j\le\ell} \cup \{x<_2b_j\}_{j>\ell}$$ is
consistent.
%
%
%
\end{itemize}

For this, it is enough to find for every $m<\omega$, sequences $\langle a_i
\rangle$ and $\langle b_j \rangle$ for $i,j<m$ such that for every $k,\ell<m$ we have:
\begin{itemize}
\item
The set $$p(x)\cup \{x>_1a_i\}_{i\le k} \cup \{x<_1a_i\}_{i>k}\cup \{x>_2b_i\}_{i\le\ell} \cup \{x<_2b_j\}_{j>\ell}$$ is
consistent.
\end{itemize}

Such $a_i$ and $b_j$ can be found by the definition of a model
companion.

This implies that every type in this theory has dp-rank at least 2, and in
particular that there are no dp-minimal types. On the other hand, it is easy to see that no one-type (over any set) has dp-rank bigger than 2.
\end{example}

\subsection{Type of dp-rank $\omega$ in a theory with types of finite weight}\label{sub:example}

People have asked whether or not strong dependence was
equivalent to every type having finite dp-rank, in the same way
that a stable theory is strongly dependent if and only if every
type has finite weight (some people asked this for Adler's notion of burden \cite{Ad2}, which is essentially the same question). Specifically, the question is whether or not it is
possible to have randomness patterns of arbitrarily large finite
depths but no randomness pattern of infinite depth for a given
complete type $p(x)$ (once one forgets the type and just asks whether there is a strongly dependent theory with arbitrarily deep randomness patterns, the question becomes much easier).

The following provides an example that the above is
possible even in stable theories.

Let
\[
S:=\{(m,n)\in \mathbb N\times \mathbb N\mid m<n\}
\]
and let $<_S$ (which we sometimes denote simply by $<$ for simplicity) be the partial order on $S$ defined by $(m_1,
n_1) <_S (m_2, n_2) $ if and only if $n_1< n_2$.

If $s=(m,n)$, we say that $s$ is \emph{of level} $n$, and write
$\lev(s)=n$. Note that $s\le t$ if and only if $s$ is of a smaller
level than $t$ (or $s=t$). So $s\neq t$ are incomparable if and
only if they are of the same level. Hence there are finitely many
$t$'s which are incomparable to a given $s$ -- in fact, the number
is exactly the $\lev(s)$.

So we have:

\begin{observation}\label{counting}
The following hold for $(S,\leq_S)$.
\begin{enumerate}
\item For any $s\in S$ there are finitely many $s'\in S$ which are
not greater than $s$.

\item For any $n\in \mathbb N$ there exists $s_1, \dots, s_n\in S$
such that $s_i$ is incomparable to $s_j$ for all $1\leq i\neq
j\leq n$.
\end{enumerate}
\end{observation}

Now consider the theory $T_\forall$ in the language $\mathcal
L:=\{E_s\}_{s\in S}$ which states that all $E_s$ are equivalence
relations and for any $s\leq_S t$ we have
\[
\forall x,y,\  x E_s y \Rightarrow x E_t y.
\]

Let $T$ be the model completion of $T_\forall$, so in particular it
is a complete theory with elimination of quantifiers.

\medskip

$T$ is axiomatized by the following axioms. In
order to make things more uniform, let us refer to equality as the
unique $s$ of level -1.

\begin{itemize}
   \item Every class of every $E_s$ for $s\in S$ is infinite.
   \item Every $E_s$ has infinitely many classes.
   \item Whenever $-1 \le n < m$, $\lev(s)=n$, $\lev(t_1)=\lev(t_2)=\ldots=\lev(t_k)=m$, $A_i$ is an equivalence class of $E_{t_i}$ (for $i=1,\ldots,k$), then $\cap_{i=1}^kA_i$ contains infinitely many classes of $E_s$.
\end{itemize}

Even without proving the existence of the model completion,
one can show directly, given the above axioms and using a standard back and
forth argument, that $T$ is a complete theory with elimination of
quantifiers.

\begin{claim}\label{omegastable}
$T$ is $\aleph_0$-stable.
\end{claim}
\begin{proof}

For every (non-generic) type $p$ over a countable model $M$, let
$\lev(p)$ be the least $n$, such that there exists $a\in M$ and
$s$ of level $n$, such that the formula $xE_s a$ is in $p$. Since
this determines which $E_t$ classes $p$ ``chooses'' for $s<t$,
there are only finitely many equivalence relations left to
``settle'' -- specifically, all those $E_t$ for which $s\not<t$
(see Observation \ref{counting} above). It follows (by quantifier
elimination) that there are countably many types of each    given
level $n$. Since there are countably many levels, we are done.
\end{proof}

We can now show that $T$ is an example of a theory with non finite
dp-rank and no randomness pattern of depth $\omega$ for the
unique type $p(x)$ over the empty set.

\begin{theorem}\label{thm:example}
For any element $c$ and any natural number $n$, one can find
$\mathcal I$ such that $\mathcal I$ is a set of $n$ mutually
indiscernible sequences (over $\emptyset$) none of which is
indiscernible over $c$. However, no such example can be found with
$\mathcal I$ infinite.
\end{theorem}
\begin{proof}
Let $n$ be a natural number, and let $s_1, \ldots, s_n$ be
incomparable elements of $S$ of level $n$ (see Observation
\ref{counting}). By the axioms of $T$ we can find $I_1$, \ldots,
$I_n$ mutually indiscernible sequences of singletons with the
following properties:

    \begin{itemize}
        \item Elements of $I_i$ are $E_{s_j}$-equivalent if and only if $j\neq
        i$,
        \item $c$ is $E_{s_i}$-equivalent to the first element of $I_i$ for all $i$.
    \end{itemize}

    Clearly, none of the $I_i$'s is indiscernible over $c$, as required.

On the other hand, by Claim \ref{omegastable}, $T$ is superstable,
hence strongly dependent (this is very easy to see directly, but
see \cite{Sh863}). So there is no infinite randomness pattern for
any type in any model of $T$. (Alternatively, using quantifier
elimination, one can easily give a direct proof that there is no
infinite randomness pattern.)
\end{proof}

\begin{remark}\label{rmk:example}
    In a similar fashion, one constructs an example of a theory which is not superstable, but still strongly dependent, with a type of dp-rank $\omega$. This is done by switching the ``nesting order'' of the equivalence relation in the example above; that is, we demand in the universal theory that for $s\leq_S t$ we have
    \[
    \forall x,y,\  x E_t y \Rightarrow x E_s y.
    \]

    Here one has to give a direct argument as to why there is no infinite randomness pattern; but this quite straightforward, using quantifier elimination, and we leave it to the reader.
\end{remark}

Recall that in a strongly stable theory (strongly dependent and
stable) every type has finite weight (\cite{Ad2,Us}). However, in
both examples discussed above, the unique (non-algebraic) type
over $\0$ has dp-rank $\omega$. So we have two examples of
theories where every type has finite weight but there are
\emph{no} non algebraic types of finite dp-rank over $\0$. In
fact, the generic type over \emph{any} set will also have dp-rank
$\omega$ by a similar argument. In the example discussed in Remark
\ref{rmk:example} the situation is even more extreme, because
almost \emph{all} types have infinite dp-rank (besides the
algebraic and the strongly minimal ones).

\section{Extending indiscernible sequences} 
\label{sec:extending_indiscernible_sequences}

Since the main definition of the paper involves mutually
indiscernible sequences, it would be nice to have certain tools
for handling such ``independent arrays''. Specifically, we would
like to have a ``consistent'' way of extending indiscernible
sequences. Our technique will make use of Shelah's notion
of average types with respect to ultrafilters, which provides a
way of constructing co-heir extensions. In spite of its usefulness,
this notion does not yet seem to be widespread in the model
theoretic community.

\medskip

First, we recall the following easy observation.

\begin{fact}\label{unique complete}
Given a set $A$ and an infinite indiscernible sequence $\mathcal I
= \langle a_i \colon i \in I \rangle$ over $A$, there exists a
unique complete type $p$ over $AI$ such  that if $a\models p$,
then $I^\frown\langle a \rangle$ is indiscernible over $A$.
\end{fact}

This fact provides a natural and standard way to extend
indiscernible sequences. However, it will not always be good
enough for us, since, as mentioned in the introduction, we will need to
extend sequences preserving indiscernibility over different subsets. 

This is why the definition of average types from Shelah becomes
useful:

\begin{definition}
Let $I=\langle a_i\rangle$ be an indiscernible sequence and
$\mathcal U$ an ultrafilter on the index set of $I$. Given any set
$B$ we will define $Avg_{\mathcal U}(I, B)$, the \emph{average type
of $I$ over $B$ given by $\mathcal U$}, as the unique complete type
$p(x)$ such that for every formula $\phi(x,y)$ and $b\in B$ we have
\[
\phi(x,b)\in p(x) \Leftrightarrow \{i \mid \phi(x,a_i)\}\in
\mathcal U.
\]
\end{definition}

This definition will allow us to prove the following. The proof
requires some knowledge of ultrafilters.

\begin{proposition}\label{unique complete 2}
Let $I$ be an indiscernible sequence indexed by an order with no
last element, and let $B$ be any set. Then, for any indexing set
$\lambda$, there is an indiscernible sequence $I^*$ indexed by
$\lambda$ such that $I^\frown I^*$ is indiscernible over $A$
whenever $A\subseteq B$ is a set such that $I$ was already
indiscernible over $A$.
\end{proposition}

\begin{proof}
We will prove the result for any finite $I^*$. The general case
follows by compactness.

If we take $\mathcal U$ to be an ultrafilter over $I$ such that
every set in $\mathcal U$ is unbounded in $I$, it follows from the
definition that for any $A$ such that $I$ is indiscernible over
$A$ and any $a_0\models Avg_{\mathcal U}(I, AI)$ we have $I^\frown
\langle a \rangle$ is indiscernible over $A$. Since the
definitions imply that $Avg_{\mathcal U}(I, BI)$ extends
$Avg_{\mathcal U}(I, AI)$ for any such $A\subset B$, any
realization $a_0$ of the type $p(x):= Avg_{\mathcal U}(I, BI)$
will be such that $I^\frown \langle{a_0}\rangle$ is indiscernible
over $A$. Inductively, it follows from the properties of average
types that if we let
\[
a_{n+1}\models Avg_{\mathcal U}(I, BI^\frown\langle a_n, a_{n-1},
\dots, a_0\rangle)\] then $I^\frown\langle a_{n+1}, a_n, a_{n-1},
\dots, a_0\rangle$ is an indiscernible sequence over $A$. This
construction will therefore give us an indsicernible sequence
$I'$, indexed by $n$, such that $I^\frown I^*$ is
indiscernible over any $A$ over which $I$ was indiscernible (for an arbitrarily large $n$); by
compactness, given any indexing set $\lambda$, we can find an
indiscernible sequence $I^*$ indexed by $\lambda$ such that
$I^\frown I^*$ satisfies the conclusion of the proposition.
\end{proof}

Proposition \ref{unique complete 2} is the only instance where we
will use average types. If unwilling to think about ultrafilters,
the reader can just assume the existence of a way to extend
indiscernible sequences given by Proposition \ref{unique complete
2}.

\begin{corollary}\label{average}
Let $I$ and $J$ be infinite mutually indiscernible sequences over
$A$ and let $B\supset A$.
Let $I^*$ be as in Proposition \ref{unique complete 2}. 

Then $I^\frown I^*$ and $J$ are mutually indiscernible over $A$.
\end{corollary}

\begin{proof}
It follows from Proposition \ref{unique complete 2} that $I^\frown
I^*$ is indiscernible over $AJ$. Now, if $J$ was not indiscernible
over $I^\frown I^*$ there would be a finite tuple $\bar a^\frown
\bar b$ with $\bar a\in I$ and $\bar b\in I^*$ such that $J$ is
not indiscernible over $A\bar a \bar b$. Since $I^\frown I^*$ was
indiscernible over $AJ$, and since $I$ is infinite, we know that
there are some $\bar a', \bar b'\in I$ such that
\[
\tp(\bar a' \bar b'/AJ)=\tp(\bar a \bar b/AJ).
\]
But this would imply that $J$ is not indiscernible over $AI$,
contradicting our hypothesis.
\end{proof}

%


\section{additivity of the dp-rank}\label{2}

In this section we will prove the (sub-)additivity of the dp-rank
(Theorem \ref{MainTheorem}), which is the main result of this
paper.


\subsection{Warm up case: dp-minimal}

The first technical lemma essentially deals with sub-additivity
for dp-minimal types. It will also form the induction base for the
general case. Although we could modify the proof of the general
statement slightly so that it deals with rank 1 as well, we
decided to include the simple base case explicitly, since it
exemplifies the general technique that we are using.


\begin{lemma}\label{case1}
Let $a$ be any tuple such that $\tp(a/A)$ is dp-minimal, let
$B\supset A$, and let $\mathcal I$ be a set of mutually
$B$-indiscernible sequences. Then for any $n$, given any $n+1$
mutually $B$-indiscernible sequences in $\mathcal I$ at least $n$
of them are mutually indiscernible over $Ba$.
\end{lemma}

\begin{proof}
We will do an induction on $n$. Since any extension of a
dp-minimal type is dp-minimal (or algebraic), if $n=1$ there is
nothing to prove.

Assume now that $\mathcal I:=\{I_1, \dots ,I_{n+1}\}$ is a set of
mutually $B$-indiscernible sequences for $B\supset A$. By
definition $\{I_1, \dots ,I_n\}$ are mutually indiscernible over
$BI_{n+1}$ so we can, by the induction hypothesis, find $n-1$ of
the $I_j$'s which are mutually indiscernible over $BI_{n+1}a$; we
may assume without loss of generality that $\{I_1, \dots,
I_{n-1}\}$ are mutually indiscernible over $BI_{n+1}a$. If
$I_{n+1}$ was indiscernible over $\{ a \}\cup B\cup \bigcup \{
I_1, \dots, I_{n-1}\}$ the sequence $\{I_1, \dots, I_{n-1},
I_{n+1}\}$ would satisfy the conditions of the claim, so we may
assume that this is not the case. Since non indiscernibility can
be witnessed by a finite sequence, we will assume for the rest of
the proof that $I_{n+1}$ is not indiscernible over $Ba\bar b$ for
some $\bar b\in \bigcup \{I_1, \dots, I_{n-1}\}$, and that $\{I_1,
\dots, I_{n-1}\}$ are mutually indiscernible over $I_{n+1}Ba$.

\begin{clm}\label{the claim}
It is enough to prove Lemma \ref{case1} under the assumption that
$I_{n+1}$ is not indiscernible over $Ba$.
\end{clm}

\begin{proof}
For each $k$ with $1\leq k<n$ we will inductively define a
``continuation'' $I_k^*$ of $I_k$ in the following way:

Suppose we have picked $I_j^*$ for $j<k$, and define $I^*_k$ be a
sequence indexed by $\omega$ as in Proposition \ref{unique
complete 2} for $I=I^k$ and $B=B\cup \bigcup_{i=1}^n
I_i\cup\bigcup_{j=1}^{k-1} I_j^*\cup \{a\}$ so that given any
subset $A\subseteq B\cup \bigcup_{i=1}^n
I_i\cup\bigcup_{j=1}^{k-1} I_j^*\cup \{a\}$, the sequence
$I_k^\frown I_k^*$ is indiscernible over $A$ whenever $I_k$ was
indiscernible over $A$.

It follows from the construction and Corollary \ref{average} that

\begin{itemize}
\item $\{I_1{}^\frown I_1^*, \dots, I_{n-1}{}^\frown  I_{n-1}^*, I_n,
I_{n+1}\}$ is a set of mutually $B$-indiscernible sequences,

\item $\{I_1{}^\frown I_1^*, \dots, I_{n-1}{}^\frown  I_{n-1}^*\}$ is
mutually indiscernible over $I_{n+1}Ba$, and

\item $I_{n+1}$ is not indiscernible over $Ba\bar b$ for some
$\bar b\in \bigcup \{I_1, \dots, I_{n-1}\}$.
\end{itemize}

Since the sequences in $\{I_1{}^\frown I_1^*, \dots,
I_{n-1}{}^\frown I_{n-1}^*\}$ are mutually indiscernible over
$I_{n+1}Ba$, there is an automorphism fixing $I_{n+1}Ba$ and
sending $\bar b$ to some $\bar b'\in \bigcup \{I_1^*, \dots,
I_{n-1}^*\}$. Now we have

\begin{itemize}
\item $\{I_1, \dots, I_{n-1}, I_n, I_{n+1}\}$ is a set of $B \bar
b'$-mutually indiscernible sequences,

\item $\{I_1, \dots, I_{n-1}\}$ is mutually indiscernible over
$I_{n+1}B\bar b'a$, and

\item $I_{n+1}$ is not indiscernible over $Bb'a$,
\end{itemize}
which, replacing $B$ with $B\bar b'$, is precisely the conditions
we started with plus the conclusion of the claim. Since any
$n$-subset of mutually $Bb'a$-indiscernible sequences of $\{I_1,
\dots, I_{n-1}, I_n, I_{n+1}\}$ would in particular be
$Ba$-indiscernible, the claim is proved.
\end{proof}

Now the lemma follows almost immediately. Since $\{I_2, I_3 \dots,
I_n, I_{n+1}\}$ are mutually indiscernible over $I_1B$, there
must, by induction hypothesis, be a subset of $n-1$ mutually
$I_1Ba$-indiscernible sequences. But such set cannot contain
$I_{n+1}$ since, by hypothesis given in Claim \ref{the claim},
this sequence is not (by itself) indiscernible over $Ba$. So
$\{I_2, I_3 \dots, I_n\}$ are mutually indiscernible over $I_1Ba$.
In exactly the same way we can prove that $\{I_1, I_3 \dots,
I_n\}$ are mutually indiscernible over $I_2Ba$ which in particular
implies that $I_1$ is indiscernible over $B\cup \{I_2, I_3 \dots,
I_n\}\cup \{a\}$. So $\{I_1,I_2, I_3 \dots, I_n\}$ are mutually
indiscernible over $Ba$, as required.
\end{proof}


\begin{corollary}(Sub-additivity of dp-rank for dp-minimal types)
    Let $\tp(a_i/A)$ be dp-minimal for $1\le i\le k$. Then the dp-rank of $\tp(a_1\ldots a_k/A)$ is at most $k$.
\end{corollary}
\begin{proof}
    By induction on $k$. For $k=1$ there is nothing to do. Assume that the Corollary holds for all sets $B$ and tuples of less than $k$ dp-minimal (over $B$) elements.

    Now fix $A$ and $a_1,\ldots,a_k$ dp-minimal over $A$. Let $I_1,\ldots, I_k,I_{k+1}$ be mutually indiscernible over $A$. By the lemma above, without loss of generality, $I_1,\ldots,I_k$ are mutually indiscernible over $Aa_k$, call it $B$.

    Recall that extensions of dp-minimal types have rank at most 1, so we may assume that $a_1, \ldots, a_{k-1}$ are dp-minimal over $B$. Hence by the induction hypothesis, dp-rank of the tuple $a_1\ldots a_{k-1}$ over $B$ is at most $k-1$. By definition, one of the sequences $I_1,\ldots,I_k$ is indiscernible over $Ba_1\ldots a_{k-1}=Aa_1\ldots a_k$, which is exactly what we needed.

\end{proof}

\begin{remark}
Notice that in Claim \ref{the claim} we did not assume that $I_k$
is indiscernible over $I_n\cup \{a\}$. That is, we know that $I_k$
(for $k<n$) is indiscernible over $BI_{\neq k}$ and over $BI_{\neq
k,n}a$, but not necessarily $BI_{\neq k}a$.
This (and the analogue issue in Lemma \ref{the claim2}) is the
reason we could not work using just Fact \ref{unique complete},
and decided to use the extensions described by Proposition
\ref{unique complete 2}.

\end{remark}

\medskip

\subsection{The finite case}

The following proposition, from which the main result of this
section will follow easily, is a generalization of Lemma
\ref{case1}.

\begin{proposition}\label{strong dp-rank}
Let $a$ be an element such that $\tp(a/A)$ has dp-rank at most
$k$, and let $\mathcal I:=\{I_1, \dots I_m\}$ be mutually
$B$-indiscernible sequences with $m>k$. Then there is an
$m-k$-subset of $\mathcal I$ of sequences which are mutually
indiscernible over $Ba$.
\end{proposition}

To prove Proposition \ref{strong dp-rank}, we rephrase the statement
in a way that will allow us to do an easy induction. For this we
will need to following definition.

\begin{definition}
Let $\mathcal I:=\{I_1, \dots I_m\}$ be mutually $A$-indiscernible
sequences, and let $a$ be any tuple. We will say that the pair
$\mathcal I,a$ satisfies $S_{k,n}$ if the following conditions
hold:

\begin{itemize}
\item $|\mathcal I|\ge k+n$,

\item For any $B\supset A$ such that $\mathcal I:=\{I_1, \dots
I_m\}$ are still mutually indiscernible over $B$, given any
$n+k$ sequences in $\mathcal I$ at least $n$ of them remain
mutually indiscernible over $Ba$.
\end{itemize}
\end{definition}

So in particular, with this notation, a type $p(x)$ over $A$ has
dp-rank less than or equal to $k$ if and only if for any
realization $a$ of $p(x)$ and every set $\mathcal I$ of mutually
indiscernible sequences where $|\mathcal I|>k$, we have that
$\mathcal I, a$ satisfies $S_{k,1}$.


With this notation we can state a generalization of Proposition
\ref{strong dp-rank}, the proof of which will admit a clear
induction argument. We will start by proving the following
analogue of Claim \ref{the claim}.

\begin{lemma}\label{the claim2}Let $a$ be an element, and let
$\mathcal I$ be a set of mutually $A$-indiscernible sequences. Let
$\mathcal J$ be a subset of $\mathcal I$ and $I\in \mathcal I$ be
such that $\mathcal J$ is mutually indiscernible over $AIa$ and
such that $I$ is not indiscernible over $A\mathcal J a$. Then we
can extend $A$ to a set $B$ such that the following hold:
\begin{itemize}
\item $\mathcal I$ is mutually indiscernible over $B$.

\item $I$ is not indiscernible over $Ba$.

\item $\mathcal J$ is mutually indiscernible over $Ba$.
\end{itemize}
\end{lemma}

\begin{proof}
We will assume that $\mathcal J$ is finite, which is the case we
need for Theorem \ref{MainTheorem}. However, the general case
follows exactly in the same manner, using ordinal enumerations of
the sequences in $\mathcal J$ and transfinite induction.

We can enumerate $\mathcal J:=\{J_1, \dots, J_n\}$ and, as in the
proof of Claim \ref{the claim}, define a ``continuation''
$J_t^*:=\langle a_i^*\rangle_{i\in \omega}$ for every sequence
$J_t\in \mathcal J$ inductively (on $t$) having that $J_t^\frown
J_t^*$ is indiscernible over any
\[
C\subseteq A\cup \bigcup \mathcal I \cup\bigcup_{j=1}^{t-1}
J_j^*\cup \{a\}
\]
over which $J_t$ was already indiscernible.

Because $\mathcal J$ was mutually indiscernible over $AIa$ it
follows from Corollary \ref{average} that

\begin{itemize}
\item $\{J_1^\frown J_1^*, \dots, J_{n}^\frown J_{n}^*\}\cup
\left( \mathcal I\setminus \mathcal J\right)$ is a set of
$A$-mutually indiscernible sequences,

\item $\{J_1^\frown J_1^*, \dots, J_{n}^\frown  J_{n}^*,\}$ is
indiscernible over $IAa$, and

\item $I$ is not indiscernible over $Aa\bar b$ for some $\bar b\in
\bigcup \{J_1, \dots, J_{n}\}$.
\end{itemize}

Since $\{J_1^\frown J_1^*, \dots, J_{n}^\frown  J_{n}^*,\}$ is
indiscernible over $IAa$ there is an automorphism fixing $IAa$ and
sending $\bar b$ to some $\bar b'\in \bigcup \{J_1^*, \dots,
J_{n}^*\}$. Now we have

\begin{itemize}
\item $\mathcal I$ is a set of $A \bar b'$-mutually indiscernible
sequences,

\item $\{J_1, \dots, J_{n}\}$ is indiscernible over $IA\bar b'a$,
and

\item $I$ is not indiscernible over $Ab'a$.
\end{itemize}
Letting $B:=Ab'$ completes the claim.
\end{proof}

\begin{proposition}\label{rephrase}
Let $a$ be an element, $n$ be any natural number, and let
$\mathcal I:=\{I_1, \dots I_m\}$ be mutually $A$-indiscernible
sequences with $m\geq k+n$ such that
 $\mathcal I,a$ satisfies $S_{k,1}$. Then $\mathcal I, a$ satisfies $S_{k,n}$.
\end{proposition}

\begin{proof}
Notice that we have already proved the result assuming $k=1$. It
is enough to show that $S_{k,n}$ implies $S_{k,n+1}$, and we will
show this by induction on $n$ (for a fixed $k$).

So let $\mathcal I$ be a set of $m$ mutually $B$-indiscernible
sequences, $a$ be an element such that $\mathcal I,a$ satisfies
$S_{k,i}$ for all $1\leq i\leq n$ (so in particular, it satisfies
$S_{k,n}$ and $S_{k,1}$) and let $\mathcal I':=\{I_1, \dots
I_{k+n+1}\}$ be a subset of $\mathcal I$; we will prove that
$\mathcal I'$ contains a subset of size $n+1$ of sequences which are
mutually indiscernible over $Ba$.


Let $I_i$ be any sequence in $\mathcal I'$. Since $\mathcal I'
\setminus \{I_i\}$ is a set of $n+k$ mutually indiscernible
sequences over $BI_i$, there is a subset $\mathcal I_i$ of size
$n$ which are mutually indiscernible over $BI_ia$. If $I_i$ is
indiscernible over $B\mathcal I_i a$ then we would have a set of
size $n+1$ of mutually indiscernible sequences over $Ba$ and the
proposition would be satisfied. So we may assume towards a
contradiction that for every $i$ the sequence $I_i$ is not
indiscernible over $B\mathcal I_ia$.

Now, for each $i$ we apply Lemma \ref{the claim2} extending $B$
until we get $I_i$ not indiscernible over $Ba$ for all $i$ and
$\mathcal I'$ are mutually indiscernible over $B$. This contradicts
$S_{k,1}$ of $\mathcal I$  (and $S_{k,n}$ too).
\end{proof}

This completes the proof of Proposition \ref{strong dp-rank}.

\begin{theorem}\label{MainTheorem}
Let $a_1,a_2$ be tuples such that $\rkdp(\tp(a_i/A))\leq k_i$ for
$i\in \{1,2\}$. Then $\rkdp(\tp(a_1, a_2/A))\leq k_1+k_2$.
\end{theorem}

\begin{proof}
Let $\mathcal I:=\{I_1,\dots, I_{k_1+k_2+1}\}$ be mutually
$A$-indiscernible sequences. By Proposition \ref{strong dp-rank}
applied to $a_1, \mathcal I$, there is a subset $\mathcal I_1$ of
$\mathcal I$ of size $k_2+1$ of sequences which are mutually
indiscernible over $Aa_1$. By definition of dp-rank of
$\tp(a_2/Aa_1)$, we get that there is a sequence $I'\in \mathcal
I_1$ which is indiscernible over $Aa_1a_2$. By definition of
dp-rank, this completes the proof of the theorem.\end{proof}

We get the following corollary (compare with Corollary \ref{strong
dependence variables}).

\begin{corollary}
Let $T$ be any theory.

If all the one variable types have finite dp-rank, then every type
(with finitely many variables) in the theory has finite dp-rank.

If all the one variable types have dp-rank $\le k$, then every
type (with finitely many variables) $p(x)$ has dp-rank $\le
|x|\cdot k$.
\end{corollary}

The following follows immediately from Theorem \ref{MainTheorem}
and Proposition \ref{thm:rank_counting}.

\begin{corollary}
Let $T$ be any theory, and assume that all the one variable types
have dp-rank $\le k$. Then for every formula $\ph(x,y)$ we have
$\alt(\ph(x,y))\le 2k|x|+1$.

In particular, if $T$ is dp-minimal, then for every $\ph(x,y)$ we
have $\alt(\ph(x,y))\le 2|x|+1$.
\end{corollary}

\medskip

\subsection{The infinite case}


The proof of sub-additivity for the infinite case is in fact much easier than in the finite
case.

\begin{theorem}\label{infinite}
Let $\mathcal I$ be a set of $\kappa$ mutually indiscernible
sequences over $A$ and $a$ an element such that $S_{\kappa, 1}$
holds for $\mathcal I,a$. Then $S_{\kappa, \kappa}$ holds for
$\mathcal I,a$. In particular for any cardinal numbers $\kappa,
\lambda$ and any tuples $a,b$ we have that
\[
\rkdp(ab/A)\leq \max(\kappa, \lambda),
\]
whenever $\rkdp(a/A)\leq \kappa$ and $\rkdp(b/A)\leq \lambda$.
\end{theorem}

\begin{proof}
Let $\mathcal I, a$ be any pair satisfying $S_{\kappa, 1}$. We can
partition $\mathcal I=\bigcup_{\alpha\in \kappa} \mathcal
I^\alpha$ into a disjoint union of $\kappa$ many sets of $\kappa$
many sequences. By hypothesis we know that for any $\beta$ the set
$\mathcal I^\beta $ are mutually indiscernible over $A\cup
\bigcup_{\alpha\neq \beta} \mathcal I^\alpha$ so by assumption we
have that some sequence $I^\beta$ in $\mathcal I^\beta$ is
indiscernible over $A\cup \bigcup_{\alpha\neq \beta}\mathcal
I^\alpha \cup \{a\}$. Doing this for any $\beta\in \kappa$ we get
a set of sequences $\{I^\alpha\}_{\alpha\in \kappa}$ such that
$I^\beta$ is indiscernible over $A\cup \bigcup_{\alpha\neq \beta}
I_\alpha \cup \{a\}$ which by definition proves that $\mathcal I,
a$ satisfies $S_{\kappa, \kappa}$. The rest of the proof follows
exactly as in Theorem \ref{MainTheorem}.
\end{proof}

Since this immediately implies that $\rkdp(ab/A)\leq \omega$
whenever $\rkdp(a/A)\leq \omega$ and $\rkdp(b/A)\leq \omega$, this
theorem provides a very easy proof of the fact that strong
dependence ($\rkdp (p(x))\leq \omega$ for all $p(x)$ or,
equivalently, no randomness pattern of depth $\omega$) only needs
to be verified in one variable. Summarizing, we get the following
(very) easy corollary, which was originally proved by Shelah in
\cite{Sh863} (Observation 1.6).

\begin{corollary}\label{strong
dependence variables} A theory $T$ is strongly dependent if and
only if all the one variable types in $T$ are strongly dependent.
\end{corollary}

Also, we get a dp-rank version of Shelah's theorem that $T$ is
dependent if and only if the independence property cannot be
witnessed by $\phi(x,y)$ with $|x| = 1$. Recall that a theory is
dependent if and only if every type is dependent, that is,
$\rkdp(p(x))\leq |T|^+$ for any type $p(x)$\footnote{We are
defining a type to be dependent depending of the behavior of the
``dual'' formula. Reasons why this is the right notion can be
found in Observation 2.7 in \cite{OnUs}.}.
So the following, which follows immediately from Theorem
\ref{infinite}, is a new (and simpler) proof of Shelah's Theorem
II.4.11 in \cite{Shelahbook}.

\nocite{Sh715}

\nocite{Sh783}

\begin{corollary}
A theory $T$ is dependent (which is equivalent to every type being
dependent) if and only if all the one variable types in $T$ are
dependent.
\end{corollary}

\section{VC-density}\label{3}

Recent results by Aschenbrenner, Dolich, Haskell, Macpherson and
Starchenko \cite{ADHMS}, show that in many of the well behaved
dependent theories, the $VC$-density can be calculated and, in
many of the cases that they considered, it is linear. In this
section we define $VC$-density (or rather, a dual notion, which we
call $VC^*$-density\footnote{Given a family $\Delta(x,y)$, the
$VC^*_\Delta$-density will correspond to the $VC$-density of the
family $\Delta$ after we invert the roles of $x$ and $y$.}), and
discuss some connections between it and dp-rank.


%


For the sake of clarity, we introduce a few notations. Let $p(x)$
be a type, and $B$ a set; we denote the set of all complete types
over $B$ which are consistent with $p(x)$ by $S^{p(x)}(B)$. If
$\varphi(x,y)$ is a formula, or $\Delta(x,y)$ is a set of
formulas, we can speak of $\varphi(x,y)$-types, or
$\Delta(x,y)$-types (over $B$) consistent with $p(x)$; these will
be denoted by $S_\varphi^{p(x)}$ and $S_\Delta^{p(x)}$,
respectively.

We also remind the reader of some basic notation for asymptotic
behavior of functions. Let $f,g \colon \setN \to \setR_+$. We
write that $f=O(g)$ if for some constant $r
> 0$ for any $n$ large enough we have $f(n)\le r\cdot g(n)$.

Recall that if $f$ is a polynomial function, its order of magnitude
is completely determined by its degree. That is, if $f,g$ are
polynomials, then $f=O(g)$ if and only if $\deg(f)\le\deg(g)$.

\smallskip

\begin{definition}\label{VCdimension}
Let $\mathcal C$ be a large, saturated enough, model of $T$, let
$\Delta(x,y)$ be a finite set of formulas in the language of $T$,
and let $p(y)$ be a (partial) type over a set of parameters in
$\mathcal C$. The \emph{$VC^*$-dimension} of $p(y)$ with respect
to $\Delta$ is greater than or equal to $n$ if there is a set $A$
of size $n$ such that, for any $A_0\subset A$ there is some
$b\models p(y)$ and some $\delta(x,y)\in \Delta$ such that for any
$a'\in A$,  we have
\[
\mathcal C\models \delta(a',b) \Leftrightarrow a'\in A_0.
\]

Whenever this happens we will say that $\Delta$ \emph{shatters} $A$
with realizations of $p(y)$.
\medskip

We will say that the \emph{$VC^*$-dimension of $p(y)$ with respect
to $\Delta$ is $n$} if the $VC^*$-dimension is greater than or
equal to $n$, but not greater than or equal to $n+1$.
\end{definition}

Notice that if $\Delta$ shatters a set $A$ with respect to $p(y)\in
S_k(B)$, then every subset of $A$ is (externally) definable as

\[
\delta(\mathcal C, b)\cap A
\]
where $b$ varies among realizations of $p(y)$ (and
$\delta\in\Delta$). If $\Delta$ is a singleton (this is, if there
is a single formula $\delta(x,y)$ in $\Delta$) this is of course
equivalent to saying that $p(y)$ is consistent with $2^{|A|}$
different $\Delta$-types over $A$. So, if instead of counting
subsets we count types, we will get a notion that, although it is
not exactly the same as $VC^*$-dimension when $\Delta$ is not a
singleton, it is closely related to this notion (particularly
asymptotically). Recall that by $S_\Delta^{p(y)}(A)$ we denote the
set of all $\Delta$-types over $A$ consistent with $p(y)$. With
this notation, we can look for the largest $n$ such that there is
some set $A$ of size $n$ such that
\[
|S_\Delta^{p(y)}(A)|\ge 2^{|A|}.
\]

\smallskip

We are slowly getting to the notion of $VC$-density that we will
work with. It was proved (apparently independently by Sauer,
Shelah, and Vapnik-Chervonenkis) that if the $VC^*$-dimension of
$\Delta$ with respect to $p(y)$ is equal to $d$, then, for any $A$
of size greater than $d$, the number $\pi_{(\Delta, p)}(A)$ of
subsets of $A$ externally definable with realizations of $p$,
satisfies
\[
\pi_{\Delta, p}(A)\leq \left.|A|\choose
0\right.+\dots+\left.|A|\choose d\right. . \]

Now, $|S_\Delta^{p(y)}(A)|$ is bounded by $|\Delta|$ times
$\pi_{\Delta, p}(A)$. This implies that when the $VC^*$-dimension of
$\Delta$ with respect to $p$ is $d$, then each formula in $\Delta$
can define (using parameters in $p$), at most $O(|A|^d)$ externally
definable sets of $A$. It follows that if we vary $A$ among
increasing subsets of an infinite set, we get polynomial growth of
the number of $\Delta$-types $|S_\Delta^{p(y)}(A)|$ over $A$, and a
very natural question to ask is whether $d$ is the best bound on the
degree of the polynomial. This prompts the following definition of
$VC^*$-density of a type. We will define (adapting the notions in
\cite{ADHMS}) the $VC_\Delta^*$-density of a type $p(y)$ over a set
$C$ to be
\[
\inf\{r\in \mathbb R^{\geq 0} \mid
|S_\Delta^{p(y)}(A)|=O(|A|^r)\text{ for all finite $A\subseteq
C^{|y|}$}\}.
\]

What we formally mean by this is that there exists a
function $f \colon \setN \to \setR_+$ such that $f=O(n^r)$, and
$\card{\tS^{p(y)}_\Delta(A)}\le f(|A|)$ for all $A \subseteq C$
finite.

If in the definition above $A$ is allowed to range over all finite
sets (that is, $C = M$ for some $M\models T$ saturated enough), we
omit ``over $C$'', and simply say ``$VC^*_\Delta$-density of $p$''.

\medskip

Using our present notation, Proposition \ref{thm:rank_counting}
implies that if $\rkdp(p(y))\ge k$, then there exists a formula
$\ph(x,y)$ and an indiscernible sequence $I$ in the sort of $y$,
such that $|S^{p(y)}_\ph(I')|\ge {n \choose k}$ for every $I'
\subseteq I$ of size $n\ge k$. This of course means that
$VC^*_\ph$-density of $p(y)$ is at least $k$. In order to make
this connection between dp-rank and $VC^*$-density more precise,
we state the following proposition.

Recall (Proposition \ref{prp:equivalence}) that $\rkdp(p) \ge k$ if
and only if there is a randomness pattern $I_\al, \ph_\al$ of depth
$k$ for $p$. Below we will say that $\rkdp(p) \ge k$ is
\emph{witnessed by formulas in $\Delta$} (where $\Delta$ is a set of
formulas) if all $\ph_\al$ are in $\Delta$.

\begin{proposition}\label{VCdensity_dprank}
    Let $p(y)$ be a type over $A$ and $\Delta$ be a set of formulas
    which is closed under boolean combinations. Then the following are
    equivalent.

    \begin{enumerate}
    \item $\rkdp(p)\geq k$, witnessed by formulas in $\Delta$.


    \item There is an $A$-indiscernible sequence $I$ and some formula
    $\phi(x,y)\in \Delta$ such that $p(x)$ has $VC_\ph^*$-density at least $k$
    over $I$.

    \item There is an $A$-indiscernible sequence $I$ and some formula
    $\phi(x,y)\in \Delta$ such that $p(x)$ has $VC^*$-density bigger than $k-1$
    with respect to $\phi(x,y)$ over $I$.
    \end{enumerate}

    %
\end{proposition}
\begin{proof}
    (i) $\implies$ (ii) by Proposition \ref{thm:rank_counting}, as explained above
    (note that the formula one gets in Proposition \ref{thm:rank_counting} is a boolean combination of
    $\Delta$-formulas, hence is itself in $\Delta$), and (ii) $\implies$ (iii) is trivial.

    (iii ) $\implies$ (i).
    %
    %
    %

    Let $f\colon \setN\to\setN$ be the following function: $f(n)=\card{S^{p(y)}_\ph(I')}$ for some/every $I' \subseteq I$
    of size $n$. It follows from the assumption that $f$ is \emph{not} $O(n^{k-1})$. This means that for
    every $r>0$ there is $n$ such that $f(n)>r\cdot n^{k-1}$.






    We may assume that the order type of $I$ is $\setR$ (indeed, all we need
    is to keep $\card{S^{p(y)}_\ph(I')}$ for all $I' \subseteq I$ finite). We write $I=\inseq{a}{r}{\setR}$.

In order to continue, we will need the following definition. Given
some $c\models p$, an element $r\in \mathbb R$ will be defined to be
a ``switch point in $I$ for $c$''
if there is
some $\epsilon\in \mathbb R$ such that either
\[
\phi(c, a_{r-\delta})\Leftrightarrow \neg \phi(c, a_{r})
\]
for all $\delta<\epsilon$,
{or}
\[
\phi(c, a_{r+\delta})\Leftrightarrow \neg \phi(c, a_{r})
\]
for all $\delta<\epsilon$.

\medskip


\begin{clm}
    There is some $c\models p$ for which
    there are at least $k$ switch points in $I$.
\end{clm}


\noindent  \emph{Proof of the claim:}  Assume this is not the case.
Let $I' \subseteq I$ finite, denote $n =: |I'|$. Let $q \in
S^{p(y)}_\ph(I')$, $c\models q$. Below we refer to a cut of $I'$
(by which we mean either an element of $I'$ or an interval between
two adjacent elements in $I'$) as a ``switch cut for $c$ in
$I'$'', if the interval it induces in $I$ contains a switch point
of $c$ in $I$. Note that there are $2n+1$ cuts in $I'$, hence at
most ${2n+1 \choose {k-1}}$ choices for possible sequences of
switch cuts (by the assumption towards contradiction).

Notice that $q$ is completely determined by knowing (1) the switch
cuts for $c$ in $I'$, (2) what are the signs of $\phi(c,y)$ of
different segments \emph{between} the switch cuts of $c$ in $I'$,
and (3) the sings of the cuts themselves. Since there are at most
${2n+1 \choose {k-1}}$ choices for possible sequences of switch
cuts, at most $2^{k}$ possible sequences of signs on the segments
between the switch cuts, and at most $2^{k-1}$ possible signs for
the switch cuts, we would get that
$\card{\tS^{p(y)}_\ph(I')}=O(|I'|^{k-1})$ for all $I'$, contrary
to the assumption. $\square$

\medskip

Let $c\models p$ be such that there are at least $k$ switch points
in $c$. We can then choose increasing indices $q_0, \ldots,
q_{4k-1}$ in $\setQ$ such that for all $i<k$ the following hold:
\begin{itemize}
        \item The closed $\mathbb R$-interval $[q_{4i}, q_{4i+1}]$ contains a switch point, and  $\models
        \ph(a_{q_{4i}},c)\leftrightarrow\neg\ph(a_{q_{4i+1}},c)$, and
        \item the open $\mathbb R$-interval $(q_{4i+1}, q_{4(i+1)})$ does not contain a switch point so that in particular
        $\models \ph(a_{q_{4i+2}},c)\leftrightarrow\ph(a_{q_{4i+3}},c)$.
\end{itemize}
Now, let $J$ be the sequences of pairs of $I$
\[J=\seq{a_{q_{2i}},a_{q_{2i+1}} \colon i<2k},
\]
and let $\ps(x_1,x_2;y)$ be the formula
$\ph(x_1,y)\leftrightarrow\ph(x_2,y)$.

Clearly, $J$ is an $A$-indiscernible sequences and since $\Delta$
is closed under boolean combinations, $\ps(x_1, x_2, y)\in
\Delta$. Finally, by construction $alt^{p(y)}_A(\ps)$ is at least
$2k$ witnessed by $J$, so by Proposition \ref{thm:rank_counting},
we have $\rkdp(p) \ge k$, witnessed by $\ps \in \Delta$, as
required.
\end{proof}

\begin{rmk}
    Notice that we needed to define switch points, of which at first there may seem to
    be as many as the alternation rank. But there is the subtle
    issue that ``isolated points'' only count as one switch point,
    even though they contribute to two for the alternation rank. In
    fact, by changing the sequence and the formula, we manage to
    ensure that all those ``switch points'' happen on ``isolated'' points (now pairs),
    each of which then contributes two alternations, hence obtaining alternation rank $2k$.
\end{rmk}

Proposition \ref{VCdensity_dprank} explains why in the example of
non-integer $VC$-density presented in \cite{ADHMS} one has to work
over sets that are not indiscernible, and why over indiscernible
sequences $VC$-density becomes an integer: in this case,
$VC$-density simply equals the appropriate dp-rank.

\bigskip

We now combine Propositions \ref{prp:equivalence},
\ref{thm:rank_counting}, and \ref{VCdensity_dprank}, and summarize
all the main characterizations of finite dp-rank that we have
shown in this article.

\begin{theorem}\label{summary}
    The following are equivalent for a type $p(y)$ over a set $A$:
    \begin{enumerate}
        \item $\rkdp(p)\ge k$.
        \item There is a randomness pattern of depth $k$ for $p(x)$
        over $A$.
        \item There is a formula $\ph(x,y)$ and an $A$-indiscernible sequence $I$ in the sort of $x$ such that
        the $VC^*_\ph$-density of $p$ over $I$ is at least $k$.
        \item There is a formula $\ph(x,y)$ and an $A$-indiscernible sequence $I$ in the sort of $x$ such that
        the $VC^*_\ph$-density of $p$ over $I$ is bigger than
        $k-1$.
        \item There is a formula $\ph(x,y)$ and an $A$-indiscernible sequence $I$ in the sort of $x$ such that
        for every $I' \subseteq I$ of size $k$ there exists $c \models p$ satisfying $\ph(x,c)\cap I = I'$
        (that is, \emph{every} subset of $I$ of size $k$ is externally $\ph(x,y)$-definable by a realization of
        $p$).
        \item There is a formula $\ph(x,y)$ with $\alt^{p(y)}_A(\ph)\ge
        2k$.
    \end{enumerate}
\end{theorem}

One may obtain a more precise (but also more technical) version of
the theorem by restricting to a set of formulas $\Delta$ closed
under boolean combinations, as in Proposition
\ref{VCdensity_dprank}.

\medskip

\begin{rmk}
    In the proof of (iii) $\implies$ (i) in Proposition \ref{VCdensity_dprank} we only needed that
    $\card{S^{p(y)}_\ph(I')}$ is not $O(|I'|^{k-1})$, whereas the assumption gives more: not $O(|I'|^s)$ for some
    $s>k-1$. Along this line, we note that using clause (v) in Theorem
\ref{summary}, one can deduce that the following  statements are
also equivalent to $\rkdp(p)\ge k$. The statements are more
technical, but the equivalences are
    stronger.
    \begin{itemize}
        \item[(1)] There is $\ph(x,y)$  and $I$ such that $\card{S^{p(y)}_\ph(I')}=\Omega(|I'|^k)$ for $I'$ finite.
        \item[(2)] There is $\ph(x,y)$  and $I$ such that $\card{S^{p(y)}_\ph(I')}=\omega(|I'|^{k-1})$ for $I'$ finite.
        \item[(3)] There is $\ph(x,y)$  and $I$ such that $\card{S^{p(y)}_\ph(I')}$ is not $O(|I'|^{k-1})$ for $I'$ finite.
    \end{itemize}
Where $f=\Omega(g)$ means $g=O(f)$, whereas $f=\omega(g)$ means
that $f$ \emph{strictly} dominates $g$ up to any multiplicative
constant, that is, for every constant $r > 0$ we have $f(n)>r\cdot
g(n)$ for \emph{all} $n$ large enough.

    Note that in order to go from (3) to (1) one may need to change the formula and the indiscernible sequence
    (just like in the equivalence of (iv) and (v) in the theorem), and this is crucial. One may ask whether
    similar statements hold with the same formula and sequence. We have not given it much thought.
\end{rmk}

\medskip

We have recently learned
that Vincent Guingona and Cameron Hill have investigated
$VC^*$-density (and other properties) over indiscernible sequences
in much greater detail in \cite{GuHi}.

\bigskip

Notice that although Proposition \ref{VCdensity_dprank} demonstrates
a nice connection between dp-rank and $VC^*$-density, it is still
quite unsatisfactory. One would hope to connect $VC^*$-density
\emph{in general} to dp-rank. For example, all known example of
dp-minimal theories seem to have $VC^*$-density 1 (most of what is
known has been proved in \cite{ADHMS}). Is this a coincidence, or an
example of a deep connection? Specifically, we ask:

\begin{question}
    Does every dp-minimal theory have $VC^*$-density 1?
\end{question}

It is not so clear how to approach the general question. The proofs
in \cite{ADHMS} are very case-specific and difficult. Any statement
which states a bound for the $VC^*$-density in terms of the dp-rank,
would need to involve achieving finite indiscernible sequences,
hence require nontrivial combinatorial arguments. Some partial
results have been obtained by the authors in a subsequent work, but
not much is known in general.

Thinking about the possible arguments, it came to our attention that
things could be much more manageable if we could concentrate in
single variables; by this we mean that both definitions --of dp-rank
and $VC^*$-density-- could be made by looking at the behavior of the
realizations of the type with respect to \emph{singletons} (for
precise statements, see the two questions that follow this
discussion).
This sort of result is not uncommon in model theory: A theory is
dependent if arbitrarily large sets of \emph{elements} (not
tuples) can not be shattered; if a dependent theory is unstable
then the strict order property can be witnessed with elements,
etc. So it would not be too surprising if both $VC^*$-density and
dp-rank could be defined by just looking at the singletons. The
following question appeared in a first version of this paper.

\begin{question}\label{conjecture}
If $p(x)$ is a (partial) type over $A$ of $dp$-rank greater than
$n$, can this be witnessed by indiscernible sequences of elements?
This is, are there $I_1, \dots, I_n$ mutually $A$-indiscernible
sequences of \emph{singletons} and some $c\models p(x)$ such that
$I_j$ is not indiscernible over $Ac$ for all $1\leq j\leq n$?
\end{question}

This was proved to be false: there are theories which are not
dp-minimal but such that given any element and any two mutually
indiscernible sequences of singletons (over the empty set), at least
one of them is indiscernible over the element. However, since the
theory is not dp-minimal, you can find a type over the empty set
with dp-rank bigger than 1, thus providing a counterexample to the
question even with $A=\emptyset$.

The question, however, turned out to be the wrong question. The
following was proved by Kaplan and Simon in \cite{KaSi}:

\begin{fact}\cite{KaSi}
If $p(x)$ is a (partial) type over $A$ of $dp$-rank greater than
$n$, there is an extension $q$ of $p$ over some $B\supset A$ such
that $q$ has $dp$-rank greater than $n$, witnessed by indiscernible
sequences of singletons.
\end{fact}


The second questions is concerned with the behavior of $VC^*$-density:

\begin{question}\label{question2}
Suppose that $p(y)$ is a type such that for all
\[
\Delta(x,y):=\{\delta_1(x,y),\delta_2(x, y),\dots, \delta_n(x, y)\}
\] where $x$ is a single variable, we have that the $VC^*$-density of
$p(y)$ with respect to $\Delta$ is greater than $d$. Is $d$ the
$VC^*$-density of $p(y)$ with respect to \emph{any} $\Delta$?
\end{question}

Notice that a positive answer to Question \ref{question2} would
imply that we could define the VC-density of a type by considering
formulas $\Delta$ for which $\bar x$ is a singleton. If this were
true,
we would have more tools and evidence for establishing a tighter
connection between dp-rank and $VC^*$-density.

\nocite{Sh715}

\bibliographystyle{abbrv}
\bibliography{dp-minimality-final}

\end{document}